\documentclass[10pt]{amsart}
\usepackage{amssymb,latexsym}
\usepackage{amsmath}               
\usepackage{amsfonts}       
\usepackage{amsthm}                
\usepackage{setspace}
\usepackage{amssymb}
\usepackage{color}
\usepackage{verbatim}
\usepackage{comment}
\usepackage{graphicx}

\numberwithin{equation}{section}

\theoremstyle{plain}
\newtheorem{theorem}[equation]{Theorem}
\newtheorem{lemma}[equation]{Lemma}

\newtheorem{corollary}[equation]{Corollary}

\newtheorem{observation}[equation]{Observation}
\newtheorem{o.problem}[equation]{Question}

\theoremstyle{definition}
\newtheorem{definition}[equation]{Definition}
\newtheorem{example}[equation]{Example}
\newtheorem{remark}[equation]{Remark}

\def\C{\mathbb{C}}
\def\N{\mathbb{N}}
\def\Z{\mathbb{Z}}
\def\Q{\mathbb{Q}}

\DeclareMathOperator{\Gal}{Gal}

\DeclareMathOperator{\Mon}{Mon}
\DeclareMathOperator{\charp}{char}

\title[Functional Composition of polynomials]{Functional Composition of polynomials: Indecomposability, Diophantine equations and Lacunary polynomials}

\author{Dijana Kreso}
\address{
Institut f\"ur Analysis und Computational Number Theory (Math A)\\
Technische Universit\"at Graz\\ Steyrergasse 30/II\\
8010 Graz, Austria
}
\email{kreso@math.tugraz.at}

\author{Robert F.\@ Tichy}
\address{
Institut f\"ur Analysis und Computational Number Theory (Math A)\\
Technische Universit\"at Graz\\ Steyrergasse 30/II\\
8010 Graz, Austria
}
\email{tichy@tugraz.at}

\date{}

\begin{document}

\dedicatory{Dedicated to Professor Ludwig Reich on the occasion of his $75$th birthday}

\begin{abstract}
Starting from Ritt's classical theorems, we give a survey of results in functional decomposition of polynomials and of applications in Diophantine equations. This includes sufficient conditions for the indecomposability of polynomials, the study of decompositions of lacunary polynomials and the finiteness criterion for the equations of type $f(x)=g(y)$.
\end{abstract}


\maketitle


\section{Introduction}

In 1920's, in the frame of investigations of functional equations by the founders of modern iteration theory (Fatou, Julia and Ritt), Ritt~\cite{R22} studied equations of type
\[
f_1\circ f_2\circ \dots \circ f_m=g_1\circ g_2\circ \dots \circ g_n
\]
in nonconstant complex polynomials. This resulted in him studying the possible ways of writing a complex polynomial as a functional composition of polynomials of lower degree.

For an arbitrary field $K$, a polynomial $f\in K[x]$ with $\deg f>1$ is called \emph{indecomposable} (over $K$)
if it cannot be written as the composition $f(x)=g(h(x))$ with $g,h\in K[x]$ and $\deg g>1$, $\deg h>1$. Any representation of $f(x)$ as a functional composition of  polynomials of degree greater than $1$ is said to be a \emph{decomposition} of $f(x)$. It follows by induction that any polynomial $f(x)$ with $\deg f>1$ can be written as a composition of indecomposable polynomials -- such an expression for $f(x)$ is said to be a {\it complete decomposition} of $f(x)$. A complete decomposition of a polynomial clearly always exists, but it does not need to be unique.
Ritt showed that when $K=\C$ any complete decomposition of $f(x)$ can be obtained from any other through finitely many steps, where each step consists of replacing two adjacent indecomposable polynomials in a complete decomposition of $f(x)$ by two others with the same composition. Ritt then solved the equation $a\circ b = c\circ d$ in indecomposable polynomials $a, b, c, d\in \C[x]$. In this way, Ritt completely described the extent of non-uniqueness of  factorization of polynomials with complex coefficients with respect to functional composition.

Ritt wrote his proofs  in the language of Riemann surfaces and obtained results for polynomials over complex numbers. His results have been extended to polynomials over fields other than the complex numbers  by Engstrom~\cite{E41}, Levi~\cite{L42}, Fried and McRae~\cite{FM69}, Fried~\cite{F74}, Dorey and Whaples~\cite{DW74}, Schinzel~\cite{S82, S00}, Tortrat~\cite{T88} and Zannier~\cite{Z93}. Their results found applications in a variety of topics, see \cite{BWZ09, BT00, GTZ08, Pak08, Pak11, Z98}.

One such topic is the classification of polynomials $f, g\in \Q[x]$ such that the equation $f(x)=g(y)$ has infinitely many integer solutions. This problem has been of interest to number theorists at least since the 20's of the past century  when Siegel's classical theorem~\cite{S29} on integral points on curves appeared. The classification has been completed by Bilu and Tichy~\cite{BT00} in 2000, building on the work of Fried~\cite{F73, F734, F74} and Schinzel~\cite{S82}. Their theorem proved to be widely applicable and has served to prove finiteness of integer solutions of various Diophantine equations of type $f(x)=g(y)$, see for instance \cite{BKLP12, BBKPT02, BFLP13, BST00, DG06, DT01, KR13, KS03, KS05, PPS11, P09, R04, T03, ST03, ST05}.

In the present paper we survey polynomial decomposition results and the applications to Diophantine equations. In Section~\ref{SecFirst}, we present Ritt's Galois-theoretic framework for addressing decomposition questions. In Section~\ref{SecFried}, we explain the connections between decompositions of $f(x)$ and $g(x)$ and reducibility of $f(x)-g(y)$  and we explain the connections to Diophantine equations. In Section~\ref{SecSecond}, we focus on applications of the criterion of Bilu and Tichy. We survey methods used in the applications and we illustrate an application of the criterion by proving some new results. We give a number of remarks about sufficient conditions for the indecomposability of polynomials, which haven't been present in the literature. In Section~\ref{lacunary}, we focus on decompositions of lacunary polynomias (polynomials with few terms), which have received a special attention in the literature on polynomial decomposition. We survey recent developments in this area and we present a new result on decompositions of quadrinomials.

\section{Galois-theoretic approach to decomposition questions}\label{SecFirst}

In this section we present a framework which serves us to translate many questions about polynomial decomposition into field theoretic and group theoretic questions. For a detailed presentation see \cite[Chap.\@~1]{S00} and \cite{ZM}. We first recall two classical theorems of Ritt~\cite{R22} and we give a number of indications on how to prove these results.

\begin{theorem}\label{First}
Let $K$ be a field and let $f\in K[x]$ be such that $\charp(K)\nmid \deg f$ and $\deg f>1$. Then any complete decomposition $f=f_1\circ f_2\circ \cdots \circ f_m$ of $f(x)$ can be obtained from any other $f=g_1\circ g_2\circ \cdots \circ g_n$  through finitely many steps, where each step consists of replacing two adjacent indecomposables $h_i\circ h_{i+1}$ in a complete decomposition of $f(x)$  by two others $\hat h_i\circ \hat h_{i+1}$ that satisfy $h_i\circ h_{i+1}=\hat h_i\circ \hat h_{i+1}$ and $\{\deg h_i, \deg h_{i+1}\}=\{\deg \hat h_i, \deg \hat h_{i+1}\}$. 

In particular, $m=n$ and the sequence $(\deg f_i)_{1\le i\le n}$ is a permutation of the sequence $(\deg g_i)_{1\le i\le n}$.
\end{theorem}

Theorem~\ref{First} is known as Ritt's First Theorem, and it was first proved by Ritt~\cite{R22} for $K=\C$. A different proof was given by Engstrom \cite[Thm.~4.1]{E41} in case $K$ is an arbitrary field of characteristic zero. His proof extends at once to polynomials over any field with $\charp(K)\nmid\deg(f)$. See also \cite{B95, LN73, S00, ZM}. 

In the same paper where Theorem~\ref{First} appeared, Ritt solved (for $K=\C$) the equation
\begin{equation}\label{secondeq} 
h_i\circ h_{i+1}=\hat h_i\circ \hat h_{i+1},
\end{equation}
assuming
\begin{equation}\label{secondcondition}
\deg h_i=\deg \hat h_{i+1}\ \& \ \gcd(\deg h_i,\deg \hat h_i)=1.
\end{equation}
Solving \eqref{secondeq} assuming \eqref{secondcondition} generalizes solving \eqref{secondeq} in indecomposable polynomials (which is the problem that arises from Theorem~\ref{First}). In the sequel we explain why that is so. The trivial solutions of \eqref{secondeq} are
\[
\hat h_i=h_i\circ\ell, \quad \hat h_{i+1}=\ell^{(-1)}\circ h_{i+1},
\]
where $\ell\in \C[x]$ is a linear polynomial. Here $\ell^{\langle-1\rangle}(x)$ denotes the inverse of $\ell(x)$ with respect to functional composition (which clearly exists exactly when $\ell(x)$ is a linear polynomial). 

In fact, for $f\in K[x]$ with $\deg f>1$, we say that two decompositions $f=f_1\circ \dots \circ f_m$ and $f=g_1\circ \dots \circ g_n$ of $f(x)$ are \emph{equivalent} if $m=n$ and there are degree-one $\mu_1,\dots,\mu_{m-1}\in K[x]$ such that $f_{i}\circ \mu_i=g_{i}$ and $\mu_i^{(-1)}\circ f_{i+1}=g_{i+1}$, $i=1, 2, \ldots, m-1$. Thus the trivial solutions of  \eqref{secondeq} are those where $h_i\circ h_{i+1}$ and $\hat h_i\circ \hat h_{i+1}$ are equivalent decompositions.

When proving Theorem~\ref{First} Ritt noticed that any nontrivial solution of \eqref{secondeq} in indecomposable polynomials satisfies \eqref{secondcondition}. In fact, the following theorem, proved by Ritt ~\cite{R22} for the case $K=\C$, holds. 

\begin{lemma}\label{EQUIV}
Let $K$ be a field and let $h_i, h_{i+1}, \hat h_i, \hat h_{i+1}\in K[x]$ be such that $h_i\circ h_{i+1}=\hat h_i\circ \hat h_{i+1}=f$ and $\charp(K)\nmid \deg f$. Then the following holds.
\begin{itemize}
\item[i)] If $\deg h_i = \deg \hat h_i$, and hence $\deg h_{i+1}=\deg \hat h_{i+1}$, then $h_i\circ h_{i+1}$ and $\hat h_i\circ \hat h_{i+1}$ are equivalent decompositions.
\item[ii)] If $h_i, h_{i+1}, \hat h_i, \hat h_{i+1}$ are indecomposable \textup{(}over $K$\textup{)} then either $\deg h_i = \deg \hat h_i$ and $\deg h_{i+1}=\deg \hat h_{i+1}$ or \eqref{secondcondition} holds.
\end{itemize} 
\end{lemma}

Lemma~\ref{EQUIV} will also be of importance in Section~\ref{SecSecond}. Levi~\cite{L42} proved it for a field $K$ with $\charp(K)=0$ and his proof extends at once to arbitrary field $K$. Find a modernized version of Ritt's proof  in \cite{ZM}.

By Lemma~\ref{EQUIV} it follows that to completely solve the equation \eqref{secondeq} in indecomposables we may henceforth assume \eqref{secondcondition}. Ritt solved \eqref{secondeq} assuming \eqref{secondcondition} and without assuming indecomposability (for $K=\C$). This result is known in the literature as Ritt's Second Theorem. Zannier~\cite{Z93} proved Ritt's Second Theorem for arbitrary field $K$. The theorem follows. 

\begin{theorem}\label{Second}
Let $K$ be a field and assume that $h_i, h_{i+1}, \hat h_i,\hat h_{i+1}\in K[x]$ of degree greater than $1$ satisfy \eqref{secondeq} and \eqref{secondcondition}, and that $h_i'(x) \hat h_i'(x)\neq 0$. Let $m=\deg h_i$ and $n=\deg \hat h_i$, and assume without loss of generality that $m>n$. Then there exist linear polynomials $\ell_1, \ell_2, \mu_1, \mu_2\in K[x]$ such that one of the following holds:
\begin{align*}
\ell_1\circ h_i \circ \mu_1&=x^rP^n(x), \quad \ell_1\circ \hat h_i\circ \mu_2=x^n\\
\mu_1^{(-1)}\circ h_{i+1} \circ \ell_2&=x^n , \quad \mu_2^{(-1)}\circ \hat h_{i+1}\circ \ell_2=x^rP(x^n)
\end{align*}
where $P(x)\in K[x]$ is such that $r=m-n\deg P\geq 0$, or
\begin{align*}
\ell_1\circ h_i \circ \mu_1&=D_m(x, a^n), \quad \ell_1\circ \hat h_i\circ \mu_2=D_n(x, a^m)\\
\mu_1^{(-1)}\circ h_{i+1} \circ \ell_2&=D_n(x, a), \quad \mu_2^{(-1)}\circ \hat h_{i+1}\circ \ell_2=D_m(x, a),
\end{align*}
where $D_m(x,a)$ is is the $m$-th Dickson polynomials with parameter $a\in K$ defined by the functional equation 
\begin{equation}\label{dickson}
D_{m}\left (x+\frac{a}{x}, a\right) = x^{m} + \left(\frac{a}{x}\right)^{m}.
\end{equation}
\end{theorem}

Several authors rewrote Ritt's proof of Theorem~\ref{Second} in languages different from Ritt's (usually assuming  indecomposability of polynomials, see \cite{F952, DW74,  LN73, L42,M95, Z93, ZM}). Find different proofs of Ritt's Second Theorem in \cite{BT00, S82, T88}.

In the sequel we present a Galois-theoretic framework developed by Ritt~\cite{R22} for addressing decompositions questions. The following well known result  provides a dictionary between decompositions of $f\in K[x]$ and fields between $K(x)$ and $K(f(x))$, which then correspond to groups between the two associated Galois groups.  

\begin{theorem}[L\"uroth's theorem]\label{Luroth}
Let $K$ and $L$ be fields such that $K\subsetneq L \subseteq K(x)$, with $x$ transcendental over $K$. Then $L=K(f)$ for some $f\in K(x)$.
\end{theorem}

Theorem~\ref{Luroth} in case $K=\C$ was proved by L\"uroth in 1876, and as stated by Steinitz in 1910. For the elementary proof of L\"uroth's theorem, which relies only on Gauss's lemma on irreducibility of polynomials in $(K(x))[y]$, see \cite[Thm.\@~2.1]{Zieve} or \cite[Thm.\@~2, Sec.\@~1.1]{S00}.

\begin{remark}\label{Poly}
It is well known that if a field $L$ is as in Theorem~\ref{Luroth} and contains a nonconstant polynomial, then $f(x)$ such that $L=K(f)$ can be chosen to be a polynomial.  Find a simple proof in \cite[Thm.\@~4, Sec.\@~1.1]{S00}. See also \cite[Lemma~3.5]{Zieve}.
\end{remark}

In translations to group-theoretic questions, the relevant Galois group associated to $f(x)$ is not the Galois group of $f(x)$, but is defined as follows.

\begin{definition}\label{mon}
Let $K$ be a field. Given $f\in K[x]$ with $f'(x)\neq 0$ the \emph{monodromy group} $\Mon(f)$ is the Galois group of $f(x)-t$ over the field $K(t)$ viewed as a group of permutations of the roots of $f(x)-t$.
\end{definition}

By the Gauss's lemma  on irreducibility of polynomials it follows that the polynomial $f(x)-t$ from Definition~\ref{mon} is irreducible over $K(t)$. Indeed, $f(x)-t$ has degree one in $t$, and is hence irreducible in $K[x][t]$, which we can rewrite as $K[t][x]$. Then by the Gauss's lemma it follows that that $f(x)-t$ is irreducible in $(K(t))[x]$. Since $f'(x)\neq 0$, $f(x)-t$ is also separable. Thus $\Mon(f)$ is the Galois group of the Galois closure of $K(x)/K(f(x))$, viewed as a permutation group on the conjugates of $x$ over $K(f(x))$. Note that also $[K(x):K(f(x))]=\deg f$ since the minimal polynomial of $x$ over $K(f(x))$ is $f(X)-f(x)$.

Recall that the Galois group of an irreducible polynomial acts transitively on the set of roots of the polynomial. The monodromy group of $f(x)$ is thus a transitive permutation group. For a reminder on transitive group actions see \cite{KC}. 
The following two lemmas reduce the study of decompositions of $f(x)\in K[x]$  to the study of subgroups of any transitive subgroup of the monodromy group of $f(x)$. As we shall see in Lemma~\ref{cyclic},  for $f\in K[x]$ with $\textnormal{char}(K)\nmid \deg f$ (which is the condition in Theorem~\ref{First}) there exists a transitive cyclic subgroup of $\Mon(f)$.

\begin{lemma}\label{Decr.Incr}
Let $K$ be a field and let $f(x)\in K[x]$ be such that $f'(x)\neq 0$. Let $G$ be the monodromy group of $f(x)$ and let $H$ be a one-point stabilizer in $G$. There are bijections between the following sets:
\begin{enumerate}
\item[(1)] the set of equivalence classes of decompositions of $f(x)$,
\item[(2)] the set of increasing chains of fields between $K(f(x))$ and $K(x)$,
\item[(3)] the set of decreasing chains of groups between $G$ and $H$,
\end{enumerate}
such that the degrees of the polynomials in the decomposition in \emph{(1)}
equal the indices between successive groups in the corresponding chain in \emph{(3)}.
\end{lemma}

\begin{lemma} \label{G1}
Let $G$ be a transitive permutation group, let $H$ be a one-point stabilizer, and let $I$ be a transitive subgroup of $G$.  Then the map $\rho\colon U\mapsto U\cap I$ is a bijection from the set of groups between
$G$ and $H$  to the set of groups $J$ between $I$ and $H\cap I$ for which $JH=HJ$.
Moreover, $[G:U]=[I:U\cap I]$ and  $\rho(\langle U,V\rangle)=\langle\rho(U),\rho(V)\rangle$
and $\rho(U\cap V)=\rho(U)\cap\rho(V)$ for any groups $U,V$ between $H$ and $G$.
\end{lemma}

We omit the proof of Lemma~\ref{G1}. A simple proof can be found in \cite[Lemma~2.9]{KZ14}. See also \cite[Lemma~2.5, Cor.\@~2.6]{ZM}. For the sake of completeness we include a proof of Lemma~\ref{Decr.Incr}. A version of the proof can be found in \cite[Lemma~2.2, Cor.\@~2.3]{ZM} and \cite[Lemma~2.8]{KZ14}.

\begin{proof}[Proof of Lemma~\ref{Decr.Incr}]
Let $f=f_1\circ f_2\circ \cdots\circ f_n$ be a decomposition of $f(x)$, where $f_i\in K[x]$ for all $i=1, 2, \ldots, n$.
Associate to this decomposition the chain of fields $K(x)\supset K(f_n)\supset K(f_{n-1}\circ f_n)\supset \dots \supset K(f_1\circ f_2\circ \cdots \circ f_n)=K(f)$. 
Note that for $h_i, h_j\in K[x]$ we have that $K(h_i)=K(h_j)$ if and only if there exists $\ell\in K[x]$ such that $h_i= \ell \circ h_j$. This together with Theorem~\ref{Luroth} and Remark~\ref{Poly} shows that the chosen association yields a bijection between (1) and (2). 

Let $L$ be the splitting field of $f(x)-t$ over $K(t)$, and let $y\in L$ be a root of $f(x)-t$, so that  $t=f(y)$. Since $f'(x)\neq 0$, $f(x)-t$ is separable, and $L$ is thus a Galois extension of $K(t)=K(f(y))$ and $G=\Gal(L/K(f(y)))$. Let $\tilde H=\Gal(L/K(y))$ be the stabilizer of $y$ in $G$. Clearly, $H$ and $\tilde H$ are conjugate subgroups of $G$. Then the Galois correspondence \cite[Thm.~VI.1.1]{L02} yields a bijection between (2) and (3). The same Galois correspondence together with $[K(f_{i-1}\circ \cdots \circ f_n):K(f_i\circ \cdots \circ f_n)]=\deg f_i$ yields the last statement.
\end{proof}

\begin{lemma}\label{cyclic}
If $f\in K[x]$ is such that $\charp(K)\nmid \deg f$, then the monodromy group of $f(x)$ contains a transitive cyclic subgroup. 
\end{lemma}

The existence of a cyclic subgroup from Lemma~\ref{cyclic} is well known. One such transitive cyclic subgroup is the inertia group at any place of the splitting field of $f(x)-t$ which lies over the infinite place of $K(t)$. A proof of Lemma~\ref{cyclic} which does not require any knowledge about inertia groups is due to Turnwald~\cite{T95}. It can also be found in \cite[Lemma~6, Sec.\@~1.5]{S00}.

Note that via Lemma~\ref{Decr.Incr}, Lemma~\ref{G1} and Lemma~\ref{cyclic} the proof of Theorem~\ref{First}  reduces  to the study of subgroups of a cyclic group. Note that the transitivity of $I$ in $G=\Mon(f)$ in Lemma~\ref{cyclic} means $G=HI$, where $H$ is a one-point stabilizer in $G$ (see \cite[Thm.\@~3.12]{KC}). The first part of Theorem~\ref{First} follows by a version of Jordan-H\"older theorem (see \cite{LKZ}) about maximal chains of subgroups between $H$ and $G=HI$, which correspond to maximal chains of subgroups between $H\cap I$ and $I$ via Lemma~\ref{G1} (see \cite[Lemma~2.10]{ZM} for details). The part of Theorem~\ref{First} which involves degrees of indecomposable polynomials follows from Lemma~\ref{EQUIV} ii).

\begin{remark}
In contrast to Lemma~\ref{cyclic}, when $\charp(K)\mid \deg f$ no transitive cyclic subgroup of $\Mon(f)$ needs to exist. 
This difference is what distinguishes the case $\charp(K)\nmid \deg f$ (known as ``tame case'') from the case $\charp(K)\mid \deg f$ (known as ``wild case''). Recall that Ritt's First Theorem holds in the tame case.  Dorey and Whaples~\cite{DW74} were the first to provide an example of a polynomial $f(x)\in K[x]$ satisfying $\textnormal{char} (K) \mid \deg f$, which has two complete decompositions consisting of a different number of indecomposables. Find it also in \cite[Ex.\@~4, Sec.\@~1.3]{S00}. 
\end{remark}

It is unfortunately not possible to reduce the proof of Theorem~\ref{Second} to questions about cyclic subgroups. The known proofs of Theorem~\ref{First} and Theorem~\ref{Second} are in fact distinct. Note that if $h_i\circ h_{i+1}=\hat h_i\circ \hat h_{i+1}$ as in Theorem~\ref{Second}, then the curve $h_i(x)=\hat h_{i+1}(y)$ is irreducible (see \cite[Lemma~3, p.~23]{S00}), it admits a parametrization and its genus is zero (see for instance \cite[Thm.\@~7, p.~487]{S00}). The proof of Ritt's Second Theorem  amounts to finding all genus-zero curves of the form $h_i(x)=\hat h_{i+1}(y)$ where $h_i, \hat h_{i+1}\in K[x]$ are of coprime degrees, and thus to genus computation.  Find a detailed proof in any of \cite{BT00, S00, Z93, ZM}.

\begin{remark}\label{finer}
From Theorem~\ref{First} it follows that any two complete decompositions of $f(x)\in K[x]$ such that $\charp(K)\nmid \deg f$ and $\deg f>1$ have the same number of indecomposables and the same sequence of degrees of indecomposable polynomials. It has been subsequently shown in \cite{BN00, GS06, KZ14, M95, ZM} that two complete decompositions of $f(x)\in K[x]$ such that $\charp(K)\nmid \deg f$ share some finer invariants. For the state of the art on this topic see \cite{KZ14}.
\end{remark}

\begin{remark}
Via L\"uroth's theorem one can derive an analogue of Lemma~\ref{Decr.Incr} for rational functions with nonzero derivative. However, there's no an analogue of Lemma~\ref{cyclic} in this case.  If $K$ is a field, then  $f(x)\in K(x)$ with $\deg f>1$ (where degree of the rational function, written as the quotient of two polynomials with no common roots, is defined as the maximum of degrees of those polynomials) is called \emph{indecomposable} over $K$ if it cannot be written as the composition $f(x)=g(h(x))$ with $g,h\in K(x)$ and $\deg g>1$, $\deg h>1$. The notion of indecomposable polynomials and rational functions are compatible, see for instance \cite[Lemma~3.1]{KZ14}.
Ritt~\cite{R23} studied decompositions of rational functions over complex numbers and noted that a certain  rational function of degree $12$ can be written as both the composition of two indecomposable and of three indecomposable rational functions. This counterexample was reproduced in \cite{GS06.2}. See also the appendix of \cite{MP11} for more counterexamples. For the state of the art on invariants of rational function decomposition see \cite{KZ14}. In this paper, the authors  examined the different ways of writing a cover of curves over a field $K$ as a composition of covers of curves over $K$ of degree at least $2$ which cannot be written as the composition of two lower-degree covers. By the generalization to the framework of covers of curves, which provides a valuable perspective even when one is only interested in questions about polynomials, several improvements on previous work were made possible. 
\end{remark}

\section{Irreducibility and Indecomposability}\label{SecFried}

The importance of the monodromy group when studying various questions about polynomials was exhibited by Fried in \cite{F70, F73} in the $70$'s. See also \cite{FJ05}. Recall that to the proof of Theorem~\ref{First} of crucial importance was Lemma~\ref{cyclic} on the existence of a cyclic group of $\Mon(f)$ (when $f(x)$ has coefficients in a field $K$ such that $\charp(K)\nmid \deg f$). The following two facts are also well known. 

\begin{lemma}\label{primitive}
If $K$ is a field and $f(x)\in K[x]$ with $f'(x)\neq 0$, then $f(x)$ is indecomposable if and only if the monodromy group of $f(x)$ is a primitive permutation group.
\end{lemma}

\begin{lemma}\label{doubly}
If $K$ is a field and $f(x)\in K[x]$ with $f'(x)\neq 0$, then
\[
\phi(x, y)=\frac{f(x)-f(y)}{x-y}\in K[x, y]
\] 
is irreducible over $K$ if and only if $\Mon(f)$ is a doubly transitive permutation group.
\end{lemma}

Recall that a transitive group action is said to be primitive if it preserves no nontrivial partition of the underlying set, see \cite[Def.\@~7.11]{KC}. An action of a group $G$ on a set $X$ with $\# X\geq 2$ is called doubly transitive when, for any two ordered pairs of distinct elements $(x_1, y_1)$ and $(x_2, y_2)$ in $X^2$, there is a $g\in G$ such that $y_2=gx_2$ and $y_1=gx_1$, see \cite[Sec.\@~4]{KC}. As it is a quick proof, we recall the proof of Lemma~\ref{primitive}. Find the proof of Lemma~\ref{doubly} in \cite[Sec.~1.5, Lemma~5]{S00}. 

\begin{proof}[Proof of Lemma~\ref{primitive}]
By Lemma~\ref{Decr.Incr} it follows that $f(x)$ is indecomposable if and only if there are no proper fields between $K(f(x))$ and $K(x)$, i.e.\@ if and only if $H$ is a maximal subgroup of $G$. It is well known that a one-point stabilizer of a transitive permutation group is maximal if and only if the group is primitive, see \cite[Thm.\@~7.15]{KC}. 
\end{proof}

We remark that M\"uller~\cite{M95} classified the possible monodromy groups for indecomposable complex polynomials. The analogous problem in fields of arbitrary characteristic is not solved, and has been studied in \cite{GS95, GZ10}.

Recall also that doubly transitive actions are primitive~\cite[Cor.\@~7.17]{KC}. Also, a group action is doubly transitive if and only if it is transitive and the stabilizer of any $x\in X$ acts transitively on $X\setminus\{x\}$, see \cite[Cor.\@~4.16]{KC}. These facts together with some deeper understanding of the monodromy group of $f(x)\in K[x]$ when $\charp(K)\nmid \deg f$ (see \cite[Lemma~6, Lemma~7, p.~55--56]{S00} can be used to prove that the following theorem holds. 

\begin{theorem}\label{Fried1}
Let $K$ be a field and $f\in K[x]$ such that $\charp(K)\nmid \deg f=:n$. The following assertions are equivalent.
\begin{itemize}
\item[i)] $(f(x)-f(y))/(x-y)$ is irreducible over $\overline{K}$,
\item[ii)] $f(x)$ is indecomposable and if $n$ is an odd prime then $f(x)\neq \alpha D_n(x+b, a)+c$
with $\alpha, a, b, c\in K$, with $a=0$ if $n=3$, where $D_n(x, a)$ is the $n$-th Dickson polynomial with parameter $a$, defined by \eqref{dickson}.
\end{itemize}
\end{theorem}
Theorem~\ref{Fried1} was first proved by Fried~\cite{F70}. See further \cite{T95} or \cite[Sec.\@~1.5, Thm.\@~10]{S00} for Turnwald's (group-theoretic) proof of this result. 

In contrast to Theorem~\ref{Fried1} a simple characterization of all cases of reducibility of $f(x)-g(y)$, where $f(x),g(x)\in K[x]$, is still not known. The problem has a long history. Note that if $f(x)=\phi(f_1(x))$ and $g(x)=\phi(g_1(x))$ with $\deg \phi>1$, then $f(x)-g(y)$ is reducible over $K$. The results of Feit~\cite{Feit73}, Fried~\cite{F73} and Cassou-Nogues and Couveignes~\cite{CC99} settle the problem of reducibility of $f(x)-g(y)$ when $f(x)$ is indecomposable. When not assuming this, of importance is the following result of Fried~\cite{F73}.
\begin{theorem}\label{Fried2}
Let $K$ be a field and  $f(x), g(x)\in K[x]$ with $f'g'\neq 0$. There exists polynomials $f_1, g_1\in K[x]$ and polynomials $f_2, g_2\in K[x]$ such that
\[
f=f_1\circ f_2, \quad g=g_1\circ g_2,
\]
and that
\begin{itemize}
\item the splitting field of $f_1(x)-t$ over $K(t)$ equals the splitting field of $g_1(x)-t$ over $K(t)$, where $t$ is transcendental over $K$. 
\item for every irreducible factor $F_1(x, y)$ of $f_1(x)-g_1(y)$, the polynomial $F(x, y)=F_1(f_2(x), g_2(y))$ is irreducible,
\item every irreducible factor of $f(x)-g(y)$ is of the form $F_1(f_2(x), g_2(y))$, where $F_1(x, y)$ is an irreducible factor of $f_1(x)-g_1(y)$.
\end{itemize}
Thus
\[
F_1(x, y) \to F(x, y)=F_1(f_2(x), g_2(y))
\]
is a bijection between the irreducible factors of $f_1(x)-g_1(y)$ and $f(x)-g(y)$.
\end{theorem}

See also \cite{BT00} for more detailed exposition of Fried's proof of Theorem~\ref{Fried2}. Both references \cite{F73} and \cite{BT00} state the result for fields of characteristic $0$, but the proof extends at once to arbitrary fields. Theorem~\ref{Fried2} has important implications, see  \cite{BT00, FS72}. In particular, it has been an important ingredient in the classification of polynomials $f, g\in \Q[x]$ such that the equation $f(x)=g(y)$ has infinitely many integer solutions, as will be explained in the next section.

\section{Diophantine equations and Indecomposability}\label{SecSecond}

Ritt's polynomial decomposition results have been applied to a variety of topics. One such topic is the classification of polynomials $f, g\in \Q[x]$ such that the equation $f(x)=g(y)$ has infinitely many integer solutions. This problem has been of interest to number theorists at least since the 20's of the past century. Deep results in algebraic geometry and Diophantine approximations have been employed to address this problem. In 1929, Siegel~\cite{S29} used such methods to prove one of the most celebrated results in this area.  In the sequel we present results for polynomials with coefficients in $\Q$, as we focus on the equation $f(x)=g(y)$ with $f, g\in \Q[x]$. We start by recalling Siegel's theorem (in this special case).

Let $F(x, y)\in \Q[x, y]$ be absolutely irreducible (irreducible over the field of complex numbers). By \emph{points} of the plane curve $F(x, y)=0$ we always mean places of its function field $\overline \Q(x, y)$ (as usual, we denote by $x$ and $y$ both independent variables and coordinate functions on the plane curve). The place is \emph{infinite} if it is a pole of $x$ or $y$. The corresponding point of the plane curve is called a \emph{point at infinity}.  Genus of a plane curves is a genus of its function field. An absolutely irreducible $F(x, y)\in \Q[x,y]$ is said to be \emph{exceptional} if the plane curve $F(x, y)=0$ is of genus $0$ and has at most two points at infinity.  For $F(x, y)\in \Q[x, y]$ the equation $F(x, y)=0$ is said to have infinitely many {\it rational solutions with a bounded denominator} if there exists $\lambda \in \N$ such that $F(x, y)=0$ has infinitely many solutions $x, y\in \Q$ that satisfy $\lambda x, \lambda y \in \Z$. 

\begin{theorem}[Siegel's theorem]\label{Siegel}
Let $F(x, y)\in \Q[x, y]$ be an absolutely irreducible polynomial. If the equation $F(x, y)=0$ has infinitely many rational solutions with a bounded denominator, then the polynomial $F(x, y)$ is exceptional. 
\end{theorem}

Davenport, Lewis and Schinzel~\cite{DLS61} were the first to present a finiteness criterion for the equation $f(x)=g(y)$. They provided sufficient conditions on $f(x)$ and $g(x)$ for $f(x)-g(y)$ to be  irreducible and the corresponding plane curve of positive genus. This criterion was quite restrictive for applications. 

Fried investigated this problem in a series of papers \cite{F73, F74, F734} of fundamental importance. Write 
$f(x)=f_1(f_2(x))$ and $g(x)=g_1(g_2(x))$, where $f_1, f_2, g_1, g_2\in \Q[x]$ are as in Theorem~\ref{Fried1}. Clearly, if the equation $f(x)=g(y)$ has infinitely many rational solutions with a bounded denominator, then there exists an irreducible factor $E(x, y)$ of $f(x)-g(y)$ such that the equation $E(x, y)=0$ has infinitely many such solutions. It can be easily shown that such $E(x, y)$ must in fact be absolutely irreducible (see \cite[Sec.\@~9.6]{S93}). By Siegel's theorem it follows that $E(x, y)$ is exceptional. By Theorem~\ref{Fried1} it follows that $E(x, y)=q(f_2(x), g_2(y))$ where $q(x)$ is an absolutely irreducible factor of $f_1(x)-g_1(y)$. Since $E(x, y)$ is exceptional, it can be easily shown that $q(x, y)$ must be exceptional as well, see \cite[Prop.\@~9.1]{BT00}. Since $\deg f_1=\deg g_1$, it follows that the curve $q(x, y)=0$ has exactly $\deg q$ points at infinity. Since $q(x, y)$ is exceptional it follows that $\deg q\leq 2$. 
Thus, as pointed out by Fried~\cite{F73}, the study of Diophantine equation $f(x)=g(y)$ requires the classification of polynomials $f, g\in \Q[x]$ such that $f(x)-g(y)$ has a factor of degree at most $2$. It further requires to determine for which $f(x)$ and $g(x)$, for a given $q(x, y)$ of degree at most $2$, is $q(f(x), g(y))$ exceptional. In \cite{F74}, Fried presented a very general finiteness criterion for the equation $f(x)=g(y)$, but still not explicit.
 
Schinzel~\cite{S82} obtained a completely explicit finiteness criterion  under the assumption $\gcd (\deg f, \deg g) = 1$. If this condition holds, then $f(x)-g(y)$ is irreducible as shown by Ehrenfeucht~\cite{E58}, and by Siegel's theorem $f(x)-g(y)$ must be exceptional if the equation has infinitely many rational solutions with a bounded denominator. So, $f(x)=g(y)$ is a curve of genus $0$. In this special case, the criterion almost immediately follows from Ritt's Second Theorem. Namely, as already explained (just before Remark~\ref{finer}) the proof of Ritt's Second Theorem amounts to finding all genus-zero curves of the form $f(x)=g(y)$ with $\gcd(\deg f, \deg g)=1$. 

The classification of polynomials $f, g\in \Q[x]$ such that $f(x)-g(y)$ is exceptional has been completed by Bilu and Tichy~\cite{BT00}. It required a generalization of Ritt's Second Theorem.
The problem of classifying polynomials $f, g\in \Q[x]$ such that $f(x)-g(y)$ has a factor of degree at most $2$ was completely solved by Bilu~\cite{B99} in 1999. In 2000, Bilu and Tichy~\cite{BT00} presented a very explicit finiteness criterion which proved to be widely applicable. In what follows we recall their theorem and we discuss applications. 

\subsection{Finiteness criterion}
To state the main result of \cite{BT00}, we need to define the so called ``standard pairs'' of polynomials.  In what follows $a$ and $b$ are nonzero rational numbers, $m$ and $n$ are positive integers, $r$ is a nonnegative integer, $p \in \Q[x]$ is a nonzero polynomial (which may be constant) and $D_{m} (x,a)$ is the $m$-th Dickson polynomial with parameter $a$ defined by \eqref{dickson}. The coefficients of Dickson polynomial are given by 
\begin{equation}\label{expdick}
D_m(x,a)=\sum_{j=0}^{\lfloor m/2 \rfloor} \frac{m}{m-j} {m-j \choose j} (-1)^ja^{j} x^{m-2j}.
\end{equation}
Standard pairs of polynomials over $\Q$ are listed in the following table.

\vspace{0.3cm}
\begin{center}
\scalebox{0.8}{
 \begin{tabular}{|l|l|l|}
                \hline
                kind & standard pair (or switched) & parameter restrictions \\
                \hline
                first & $(x^m, a x^rp(x)^m)$ & $r<m, \gcd(r, m)=1,\  r+ \deg p > 0$\\
                second & $(x^2,\left(a x^2+b)p(x)^2\right)$ & - \\
                third & $\left(D_m(x, a^n), D_n(x, a^m)\right)$ & $\gcd(m, n)=1$\\
                fourth & $(a ^{\frac{-m}{2}}D_m(x, a), -b^{\frac{-n}{2}}D_n (x,b))$ & $\gcd(m, n)=2$\\
                fifth & $\left((ax^2 -1)^3, 3x^4-4x^3\right)$ & - \\
                \hline
        \end{tabular}}
\end{center}
\vspace{0.3cm}

Having defined the needed notions we now state the main result of \cite{BT00}.

\begin{theorem}\label{T:BT}
Let $f, g\in \Q[x]$ be non-constant polynomials.
Then the following assertions are equivalent.
\begin{itemize}
\item[-] The equation $f(x)=g(y)$ has infinitely many
rational solutions with a bounded denominator;
\item[-] We have 
\begin{equation}\label{BTeq}
f(x)=\phi\left(f_{1}\left(\lambda(x\right)\right)\quad \& \quad g(x)=\phi\left(g_{1}\left(\mu(x)\right)\right),
\end{equation}
where $\phi\in\Q[x]$, $\lambda, \mu\in \Q[x]$ are linear polynomials,
and $\left(f_{1}(x),g_{1}(x)\right)$ is a
standard pair over $\Q$ such that the equation $f_1(x)=g_1(y)$
has infinitely many rational solutions with a bounded denominator.
\end{itemize}
\end{theorem}

Note that if the equation $f(x)=g(y)$ has only finitely many rational solutions with a bounded denominator, then it clearly has only finitely many integer solutions. 

The proof of Theorem~\ref{T:BT} relies on Siegel's classical theorem~\cite{S29}, and is consequently ineffective (there's no algorithm for finding all solutions).

Theorem~\ref{T:BT} has served to prove finiteness of integer solutions of various Diophantine equations of type $f(x)=g(y)$, e.g.\@ when $f(x)$ and $g(x)$ are restricted to power-sum and alternating power-sum polynomials \cite{BKLP12, BBKPT02, KR13, KS05, R04}, classical orthogonal polynomials \cite{T03, ST03, S04, ST05}, certain polynomials arising from counting combinatorial objects \cite{BST00, P09}, and several other classes of polynomials (see for instance \cite{BFLP13, DT01,KS03,PPS11}).

We further mention that Theorem~\ref{T:BT} was recently slightly refined in \cite{BFLP13}. Via \cite[Thm.\@~1.1]{BFLP13} proving that the equation of type $f(x)=g(y)$ has only finitely many solutions, i.e. showing the impossibility of \eqref{BTeq}, can be made somewhat shorter than by using Theorem~\ref{T:BT}, as the number of standard pairs is reduced to three.

Proving that the equation $f(x)=g(y)$ has only finitely many integer solutions using Theorem~\ref{T:BT}, reduces to showing that polynomials $f(x)$ and $g(x)$ can not be written as in \eqref{BTeq}.
In what follows we list some methods and key ideas used in the above listed and related papers to handle this problem. 

\subsection{Theorem of Erd\H{o}s and Selfridge}
 If for nonconstant $f, g\in \Q[x]$ the equation $f(x)=g(y)$ has infinitely many integer solutions, then from Theorem~\ref{T:BT} it follows that 
\[
f(x)=\phi(\tilde f(x))\quad \& \quad g(x)=\phi(\tilde g(x))\quad \textnormal{with}\  \phi, \tilde f, \tilde g\in \Q[x].
\]
 Write $\phi(x)=\phi_k x^k+\cdots+\phi_0$, $\tilde f(x)=a_nx^n +\cdots+a_0$ and $\tilde g(x)=b_mx^m+\cdots+b_0$.
Then 
\[
f(x)=\phi(\tilde f(x))=\phi_k a_n^k x^{kn} + \cdots\quad  \& \quad  g(x)=\phi(\tilde g(x))=\phi_k b_m^k x^{km} + \cdots.
\]
Thus we have the following observation.

\begin{observation}\label{obs}
If \eqref{BTeq} holds for nonconstant $f, g\in \Q[x]$ and $k=\deg \phi$, then the quotient of the leading coefficients of $f(x)$ and $g(x)$ is a $k$-th power of a rational number. 
\end{observation}

The following theorem was proved by Erd\H{o}s and Selfridge~\cite{ES74} in 1974.
\begin{theorem}\label{ES}
The equation
\[
x(x+1)\cdots(x+k-1)=y^l
\]
has no solutions in integers $x>0$, $k>1$, $l>1$, $y>1$. 
\end{theorem}

In several applications of Theorem~\ref{T:BT}, Observation~\ref{obs} and Theorem~\ref{ES} were key ingredients. We mention \cite{BST00, P09} in which the finiteness of integer solutions is established for certain Diophantine equations arising from counting combinatorial objects. To illustrate how these three ingredients can be successfully combined we prove the following.

\begin{theorem}\label{hyperplanes}
For $m\geq 4$ and $n\geq 3$ the equation
\[
{x \choose 0}+{x \choose 2}+\cdots+{x \choose m}=y(y+1)\cdots(y+n-1)
\]
has only finitely many integer solutions $x$ and $y$.
\end{theorem}

\begin{proof}
Let 
\[
 H_m(x)={x \choose 0}+{x \choose 2}+\cdots+{x \choose m}, \quad R_n(x)=x(x+1)\cdots(x+n-1).
\]
Assume that the equation $H_m(x)=R_n(y)$ has infinitely many integer solutions. Then by Theorem~\ref{T:BT} it follows that
\[
H_m(x)=\phi\left(f_{1}\left(\lambda(x\right)\right)\quad \& \quad R_n(x)=\phi\left(g_{1}\left(\mu(x)\right)\right),
\]
where $\phi\in\Q[x]$, $\left(f_{1}(x),g_{1}(x)\right)$ is a
standard pair over $\Q$, and $\lambda, \mu\in \Q[x]$ are linear polynomials. Let $k=\deg\phi$. From Observation~\ref{obs} it follows that $m!$ is a $k$-th power of a rational number, and thus of an integer. From Theorem~\ref{ES} it follows that $k=1$ and hence
\begin{equation}\label{hypeq}
H_m(a_1x+a_0)=e_1f_{1}(x)+e_0\quad \& \quad R_n(b_1x+b_0)=e_1g_{1}(x)+e_0,
\end{equation}
for some $a_1, a_0, b_1, b_0, e_1, e_0\in \Q$ such that $a_1b_1e_1\neq 0$. Write
\begin{align*}
H_m(a_1x+a_0)&=\sum_{k=0}^m c_k x^k, \quad R_n(b_1x+b_0)=\sum_{j=0}^n d_j x^j.
\end{align*}
Then one easily finds that
\begin{align*}
c_m&=\frac{a_1^m}{m!}, \quad c_{m-1}=\frac{a_1^{m-1}(2a_0-m+3)}{2(m-1)!},\\
c_{m-2}&=\frac{a_1^{m-2}(3m^2-(19+12a_0)m+12a_0^2+36a_0+50)}{24(m-2)!},\\
 c_{m-3}&=\frac{a_1^{m-3}(-m^3+(6a_0+10)m^2-(12a_0^2+38a_0+53)m+8a_0^3+36a_0^2+100a_0+144)}{48(m-3)!},\\
\intertext{and}
d_n&=b_1^n, \quad d_{n-1}=\frac{b_1^{n-1}n(2b_0+n-1)}{2},\\
d_{n-2}&=\frac{b_1^{n-2}n(n-1)(3n^2+(12b_0-7)n+12b_0^2-12b_0+2)}{24}.
\end{align*}
(Compare with  \cite{P09} and \cite{BBKPT02}, where these coefficients also appeared).

If $(f_1(x), g_1(x))$ is a standard pair of the second kind, then by \eqref{hypeq} either $m=\deg f_1=2$ or $n=\deg g_1=2$. Since $m, n\geq 3$ by assumption, this can not be. If $(f_1(x), g_1(x))$ is a standard pair of the fifth kind, then $g_1(x)=(ax^2-1)^3$ or $g_1(x)=3x^4-4x^3$. In both cases $g_1'(x)$ has multiple roots. Note that by \eqref{hypeq} it follows that $b_1R_n'(b_1x+b_0)=e_1g_{1}'(x)$. Since $R_n(x)$ has $n$ simple real roots, the derivative $R_n'(x)$ has $n-1$ simple real roots, and hence so does $b_1R_n'(b_1x+b_0)$, a contradiction.
If $(f_1(x), g_1(x))$ is a standard pair of the first kind, then either $H_m(a_1x+a_0)=e_1x^m+e_0$ or $R_n(b_1x+b_0)=e_1x^n+e_0$. 
If $H_m(a_1x+a_0)=e_1x^m+e_0$, then since $m>3$ it follows that $c_{m-1}=0$, $c_{m-2}=0$ and $c_{m-3}=0$. From the first two identities we have $2a_0=m-3$ and $a_1^{m-2}(m-23)=0$. Thus $m=23$ and $a_0=10$. Then $c_{20}=a_1^{20}/20!\neq 0$, a contradiction. If $R_n(b_1x+b_0)=e_1x^n+e_0$, then since $n\geq 3$ it follows that $d_{n-1}=0$ and $d_{n-2}=0$. From the former it follows that $2b_0=1-n$ and by substituting this into $d_{n-2}=0$ we get $n\in \{-1, 0, 1\}$, a contradiction. 
If $(f_1(x), g_1(x))$ is a standard pair of the third or of the fourth kind, then either $H_m(a_1x+a_0)=e_1D_m(x, a^n)+e_0$ or $H_m(a_1x+a_0)=e_1a^{-m/2}D_m(x, a)+e_0$ or $H_m(a_1x+a_0)=-e_1a^{-m/2}D_m(x, a)+e_0$ for some nonzero $a\in \Q$. In all cases it follows that $c_{m-1}=0$ and $c_{m-3}=0$, since \eqref{expdick} holds and $m\geq 4$. The former implies $2a_0=m-3$, wherefrom $c_{m-3}=a_1^{m-3}/(m-3)!\neq 0$,  a contradiction.
\end{proof}
\begin{remark}
The polynomial on the left hand side of the equation in Theorem~\ref{hyperplanes}, denoted by $H_m(x)$  in the proof, has a combinatorial interpretation, see \cite{P09}. There the author studied the  equation $H_m(x)=H_n(y)$ with $m>n\geq 3$. Note that the same strategy as in the proof of Theorem~\ref{hyperplanes} applies.
\end{remark}

\subsection{Indecomposability criteria}

A standard way to examine the finiteness of solutions of an equation of type $f(x)=g(y)$ with nonconstant $f, g\in \Q[x]$ is to first find the possible decompositions of $f(x)$ and $g(x)$, and then compare those with \eqref{BTeq}. For simiplicity, we write in this section ``indecomposable'' when we mean ``indecomposable over complex numbers''.  In \cite{DG06} and \cite{DGT05} sufficient conditions for $f(x)\in \Q[x]$ to be indecomposable are studied. In what follows we recall these and related results and give some further remarks on indecomposability of polynomials. 

For simplicity we first restrict our attention to monic polynomials. Let monic $f\in \Q[x]$  be decomposable, i.e.\@ there exist $g, h\in \C[x]$ such that $\deg g>1$ and $\deg h>1$ and $f(x)=g(h(x))$. Note that we may assume that $g(x)$ and $h(x)$ are monic as well and that $h(0)=0$ since we can clearly find an equivalent decomposition which satisfies these assumptions. Indeed, if $f(x)=\tilde g(\tilde h(x))$ and $a$ is the leading coefficient of $\tilde h(x)$, then also
\[
f(x)=\tilde g(ax+\tilde h(0))\circ \frac{1}{a}(\tilde h(x)-\tilde h(0).
\]
Thus, if monic $f(x)\in \Q[x]$ is decomposable, we may write without loss of generality
\begin{equation}\label{monic}
\begin{cases}
f(x)=g(h(x))\\
f(x)=x^n +c_{n-1}x^{n-1}+\cdots+c_0\\
g(x)=x^{t}+a_{t-1}x^{t-1}+\cdots+a_0, \\
 h(x)=x^{k}+b_{k-1}x^{k-1}+\cdots +b_0,\  \textnormal{with}\ b_0=0, \\
k, t\geq 2, \ a_i, b_j\in \C,\  i=0, 1, \ldots, t-1,\  j=0, 1, \ldots, k-1.
\end{cases}
\end{equation}
It follows that 
\begin{equation}\label{inner}
f(x)= h(x)^t+a_{t-1} h(x)^{t-1}+ \cdots +a_1h(x)+a_0.
\end{equation}

Note that $\deg (h(x)^{t-1})=n-k$, so we can compare the first $k-1$ coefficients of $f(x)$ (starting from the leading coefficient) with the corresponding coefficients of $h(x)^t$. With notation from \eqref{monic} we have the following:
\[
(*)\begin{cases}
 c_{n-1}=t b_{k-1} \\ c_{n-2}=tb_{k-2}+{t \choose 2}b_{k-1}^2 \\ \quad \quad \ \vdots \\ c_{n-k+1}=tb_{1}+\sum\limits_{i_1+2i_2+\cdots+(k-2)i_{k-2}=k-1}^{}
 d_{i_1, i_2, \ldots, i_{k-2}}\ b_{k-1}^{i_1}b_{k-2}^{i_2}\ldots b_{2}^{i_{k-2}},
\end{cases}
\]
where
\[
d_{i_1, i_2, \ldots, i_{k-2}}={t \choose {i_1, i_2, \ldots, i_{k-2}}}.
\]

\begin{lemma}\label{monic.decomp}
Assume that $f(x)\in \Q[x]$ is monic and decomposable and write without loss of generality $f(x)=g(h(x))$ as in \eqref{monic}. Then $g(x), h(x)\in \Q[x]$.
\end{lemma}
\begin{proof}
Note that from $(*)$ it follows that $b_i\in \Q$ for all $i=1, 2, \ldots, k-1$, since $c_i\in \Q$ for all $i=0, 1, 2, \ldots, n-1$. Since also $b_0=0$, it follows that $h(x)\in \Q[x]$. From \eqref{inner} it follows that $g(x)\in \Q[x]$ as well.
\end{proof}

\begin{theorem}\label{schinzel}
For any decomposition of $\tilde f(x)\in \Q[x]$, there exists an equivalent one with polynomials with coefficients in $\Q$.
\end{theorem}
\begin{proof}
Let $c\in \Q$ be the leading coefficient of $\tilde f(x)$. For any representation of monic $f(x)=(1/c) \tilde f(x)$ as a functional composition of two polynomials, there exists an equivalent decomposition $f(x)=g(h(x))$ such that $g(x)$ and $h(x)$ are monic and $h(0)=0$, as in \eqref{monic}.  Lemma~\ref{monic.decomp} completes the proof.
\end{proof}
\begin{remark}
It is well known that in Theorem~\ref{schinzel} the field of rational numbers $\Q$ can be replaced by a field $K$ such that $\textnormal{char}(K)\nmid \deg \tilde f$. Note that the proof extends at once. In other words, if $\tilde f(x)\in K[x]$ is indecomposable over $K$, then it is indecomposable over any extension field of $K$ provided  $\textnormal{char}(K)\nmid \deg \tilde f$. This was first shown by Fried--McRae~\cite{FM69} in 1969. An alternative but similar proof of Theorem~\ref{schinzel} (with $\Q$ replaced by a field $K$ such that $\textnormal{char}(K)\nmid \deg \tilde f$) can be found in \cite[Thm.\@ 6, p.20]{S00} and in \cite[Lemma 2.1]{KR13}. If $\textnormal{char}(K)\mid \deg \tilde f$, then $\tilde f(x)$ can be indecomposable over $K$, but decomposable over some extension field of $K$, see
 \cite[p.\@ 21, Ex.\@ 3]{S00}.

\end{remark}

With respect to Theorem~\ref{T:BT} of particular interest is the case of polynomials with integer coefficients. We restrict to this case in the sequel. The following was first observed by Turnwald~\cite{T95}, and  subsequently by Dujella--Gusi\'c~\cite{DG06}. We include a proof taken from the latter paper.

\begin{theorem}\label{monic.integer}
Assume that $f(x)\in \Z[x]$ is monic and decomposable and write $f(x)=g(h(x))$ as in \eqref{monic}. Then $g, h\in \Z[x]$.
\end{theorem}
\begin{proof}
From Lemma~\ref{monic.decomp} it follows that $g(x), h(x)\in \Q[x]$. Let
\[
f(x)=\prod_{i=1}^n (x-\alpha_i), \  g(x)=\prod_{j=1}^t(x-\beta_j),\ \textnormal{and hence}\ f(x)=\prod_{j=1}^t(h(x)-\beta_j)
\]
where $\alpha_i$'s are algebraic integers, and $\beta_j$'s algebraic numbers. Then
\[
h(x)-\beta_j=\prod_{l\in I_j}(x-\alpha_l),\ \textnormal{where}\ I_j\subseteq \{1, 2, \ldots, n\},
\]
because of the uniqueness of factorization over a suitable number field containing $\alpha_i$'s and $\beta_j$'s.
Since $\alpha_i$'s are algebraic integers and $h(0)=0$, it follows that $\beta_j$'s are algebraic integers as well, and hence $g(x)\in \Z[x]$ and $h(x)\in \Z[x]$.
\end{proof}

We now present the criterion from \cite{DG06}, which was also obtained as a corollary of a more general result in \cite{DGT05} about possible ways to write a monic polynomial with integer coefficients as a functional composition of monic polynomials with rational coefficients.

\begin{theorem}\label{criterion}
Let $\tilde f(x)=c_nx^n+c_{n-1}x^{n-1}+\cdots+c_0\in \Z[x]$. If $\gcd(c_{n-1}, n)=1$, then $\tilde f(x)$ is indecomposable. 
\end{theorem} 
\begin{proof}
Note that 
\[
c_n^{n-1}\tilde f(x)=(c_nx)^n+c_{n-1}(c_nx)^{n-1}+ \cdots + c_n^{n-1}c_0= f(c_nx),
\]
where $f(x)=x^n+c_{n-1}x^{n-1}+\cdots \in \Z[x]$. If $\tilde f(x)$ is decomposable, then so is $f(x)$ and we can write $f(x)=g(h(x))$ as in \eqref{monic}. From Theorem~\ref{monic.integer} it follows that $g, h\in \Z[x]$. Then from $c_{n-1}=tb_{k-1}$ in $(*)$, it follows that $\gcd (c_{n-1}, n)\geq t\geq 2$, a contradiction.
\end{proof}

\begin{remark}\label{criterion2}
Assume that monic $f\in \Z[x]$ is decomposable and write $f(x)=g(h(x))$ as in \eqref{monic}, so that $g, h\in \Z[x]$ by Theorem~\ref{monic.integer}. Clearly $n:=\deg f\geq 4$. As we have seen, a simple proof of Theorem~\ref{criterion} follows by comparison of coefficients in $(*)$. Note that from $(*)$ it also follows that 
\[
c_{n-2}=tb_{k-2}+{t \choose 2}b_{k-1}^2, \quad c_{n-3}=tb_{k-3}+{t \choose 2} b_{k-1}b_{k-2}+{t \choose 3}b_{k-1}^3.
\]
If $t>2$, it follows that $\gcd(c_{n-2}, t)>1$. Therefore if $\gcd(c_{n-2}, n)=1$ it follows that $t=2$. Assume without loss of generality that $g(x)$ is indecomposable, i.e.\@ write $f(x)$ as in \eqref{monic} with $g(x)$ indecomposable. It follows from Lemma~\ref{EQUIV} that essentially, i.e.\@ up to insertion of linear polynomials $\ell(x)$ and $\ell^{(-1)}(x)$, the only way to write $f(x)$ as a functional composition of two polynomials is 
\[
f(x)=g(x^2+bx+c), \quad \textnormal{where}\ g(x) \ \textnormal{is indecomposable}.
\]

If $t>3$, it follows that $\gcd(c_{n-3}, t)>1$. Thus if $\gcd(c_{n-3}, n)=1$ it follows that $t=2$ or $t=3$, and again by Lemma~\ref{EQUIV} it follows that essentially the only ways to write $f(x)$ as a functional composition of two polynomials are either $f(x)=g(h(x))$ where $g(x)$ is indecomposable and $\deg h=2$, or $g(x)$ is indecomposable and $\deg h=3$. 
\end{remark}

\begin{remark}\label{derivative}
Note that if $f\in \Q[x]$ is decomposable, then from Theorem~\ref{schinzel} it follows that $f(x)=g(h(x))$ for some $g, h\in \Q[x]$ with $\deg g\geq 2$ and $\deg h\geq 2$. Then $f'(x)=g'(h(x))h'(x)$. Since $\deg g'\geq 1$ and $\deg h'\geq 1$, it follows that $f'(x)$ is reducible. Thus a sufficient condition for $f(x)\in \Q[x]$ to be indecomposable is that $f'(x)$ is irreducible. 
\end{remark}

\begin{example}
Remark~\ref{derivative} can be  useful in practice. For example, it is a well-known result of Schur~\cite{Schur29} that any polynomial of type
\begin{equation}\label{schurtype}
S_n(x)=\pm 1+c_1x+c_2\frac{x^2}{2!}+\ldots +c_{n-1}\frac{x^{n-1}}{(n-1)!}\pm \frac{x^n}{n!}
\end{equation}
with $c_i\in \Z$ is irreducible in $\Q[x]$. It follows that any polynomial of type
\[
f(x)=c_0 \pm x+c_2\frac{x^2}{2!}+\dots + c_{n-1}\frac{x^{n-1}}{(n-1)!}\pm \frac{x^n}{n!}
\]
with $c_i\in \Z$ is indecomposable. Indeed, $f'(x)$ is of type \eqref{schurtype}, and is hence irreducible in $\Q[x]$. We further remark that it was shown in \cite{KS08}, by a different method,  that $f(x)$ is indecomposable in case $c_1=c_2=\cdots=c_n=1$ and $\pm$ is replaced by $+$. Thus, Remark~\ref{derivative} immediately  implies a more general result. Further note that $n!f(x)\in \Z[x]$, however neither Theorem~\ref{criterion} nor Remark~\ref{criterion2} are of help to conclude any results about possible decompositions of $f(x)$. 
\end{example}

\subsection{Critical points}\label{critical}
We now present another approach to finding possible decompositions of a polynomial.

For $f\in \C[x]$ and $\gamma\in \C$ let
\begin{equation}\label{def.delta}
\delta(f, \gamma)=\deg \gcd (f(x)-\gamma, f'(x)).
\end{equation}

\begin{lemma}\label{delta}
If $f(x)=g(h(x))$ with $\deg g>1$ then there exists $\gamma \in \C$ such that $\delta(f, \gamma)=\deg \gcd (f(x)-\gamma, f'(x))\geq \deg h$. 
\end{lemma}
\begin{proof}
If $\beta$ is a root of $g'(x)$ (which exists since by the assumption $\deg g'(x)\geq 1$) and $\gamma=g(\beta)$, then $h(x)-\beta$ divides both $f(x)-\gamma$ and $f'(x)$. 
\end{proof}
\begin{corollary}\label{gamma}
If $f\in \C[x]$ is such that $\deg f>1$ and $\delta(f, \gamma)\leq 1$ for all $\gamma \in \C$, then $f(x)$ is indecomposable.
\end{corollary}

Corollary~\ref{gamma} was first used, to the best of our knowledge, by Beukers, Shorey and Tijdeman~\cite{BST99} to prove that for arbitrary integer $m\geq 1$ the polynomial $f(x)=x(x+1)(x+2)\cdots(x+m)$ is indecomposable. It was further used  by Dujella and Tichy~\cite{DG06} to study the possible decompositions of Chebyshev polynomials of the second kind, as well as by Stoll~\cite{T03} to prove that certain classes of orthogonal polynomials are indecomposable. 

By Lemma~\ref{delta} it follows that if $f(x)\in \C[x]$ is such that $\delta(f, \gamma)\leq 2$ for all $\gamma \in \C$, then $f(x)$ is either indecomposable or $f(x)=g(h(x))$ where $\deg h=2$ and $g(x)$ is indecomposable. Compare with Remark~\ref{criterion2}. 

Note that if all the stationary points of $f(x)$ are simple and $S_f$ denotes the set of stationary points of $f(x)$, then clearly $\delta(f, \gamma)=\#\{\alpha\in S_f : f(\alpha)=\gamma\}$. In practice, sometimes the numeric evidence is obvious, but proving that no two roots $\alpha, \beta\in \C$ of $f'(x)$ are such that $f(\alpha)=f(\beta)$ is out of reach. Confer also \cite[Chap.\@~3]{T03}.

\section{Lacunary polynomials}\label{lacunary}

A polynomial with the number of terms small in comparison to the degree is called a \emph{lacunary} polynomial or \emph{sparse} polynomial (a more precise definition will not be needed). Lacunary polynomials have received a special attention in the literature on polynomial decomposition. It was shown by Schinzel that, loosely speaking, ``a power of a polynomial with many terms has many terms'', see \cite[Sec.\@ 2.6., Thm.\@ 30]{S00}. In  \cite[Chap.\@ 2]{S00}, further results on reducibility and decomposability of lacunary polynomials can be found. Schinzel also conjectured that the number of terms of $g(h(x))$ tends to infinity as the number of terms of $h(x)$ tends to infinity. This was proved in a remarkable paper by Zannier~\cite{Z08}. As a step in the proof, another result of Zannier was used, namely the following theorem from \cite{Z07}.

\begin{theorem}\label{Zannier}
Let $K$ be a field with $\textnormal{char}(K)=0$, and let $f(x)\in K[x]$ have $l>0$ nonconstant terms. Assume that $f(x)=g(h(x))$, where $g, h\in K[x]$ and where $h(x)$ is not of shape $ax^k+b$ for $a, b\in K$. Then
\[
\deg f+l-1\leq 2l(l-1)\deg h.
\]
In particular, $\deg g\leq 2l(l-1)$.
\end{theorem}
Let $K$ be a field with $\textnormal{char}(K)=0$ and $f(x)\in K[x]$ with $l>0$ nonconstant terms be decomposable and write without loss of generality
\begin{equation}\label{wlog}
f(x)=g(h(x))\  \textnormal{with}\  g, h\in K[x], \ \deg g, \deg h \geq 2,\ h(x) \ \textnormal{monic and} \ h(0)=0. 
\end{equation}
(We may indeed do so, since if $f(x)=\tilde g(\tilde h(x))$, then clearly also
\[
f(x)=\tilde g(ax+\tilde h(0))\circ \frac{1}{a}(\tilde h(x)-\tilde h(0)).
\]
where $a$ is the leading coefficient of $\tilde h(x)$.) Then Theorem~\ref{Zannier} implies that $\deg f+l-1\leq 2l(l-1)\deg h$ unless $h(x)=x^k$. 
Note that 
\[
a_1x^{n_1}+a_2x^{n_2}+\cdots+a_lx^{n_l}+a_{l+1}=f(x)=g(x)\circ x^k
\]
exactly when $k\mid n_i$ for all $i=1, 2, \ldots, l$. 

For example, let $f(x)$ be a trinomial, i.e.\@ $f(x)=a_1x^{n_1}+a_2x^{n_2}+a_3$ with $a_1a_2a_3\in K$, $a_1a_2\neq 0$, $n_1, n_2\in \N$, and $n_1>n_2$. Assume that $f(x)$ is decomposable and write it without loss of generality as in \eqref{wlog}. It follows that $\deg g \leq 3$ unless $h(x)=x^k$. 
In this case we have moreover the following stronger result proved by Fried and Schinzel~\cite{FS72} in 1972.

\begin{theorem}\label{trinom}
Let $K$ be a field with $\textnormal{char}(K)=0$. Assume that $f(x)=a_1x^{n_1}+a_2x^{n_2}+a_3$, with $a_1a_2a_3\in K$, $a_1a_2\neq 0$, $n_1, n_2\in \N$ and $n_1>n_2$, is decomposable and write without loss of generality $f(x)=g(h(x))$ with $g, h\in K[x]$ as in \eqref{wlog}. 
Then $h(x)=x^k$ and $g(x)=a_1x^{n_1/k}+a_2x^{n_2/k}+a_3$ for some $k\in \N$ with $k\mid \gcd(n_1, n_2)$.
\end{theorem}

Theorem~\ref{trinom} together with Theorem~\ref{T:BT} was used in \cite{PPS11} and \cite{S12} to show that, under some reasonable assumptions, two trinomials with rational coefficients can have equal values only finitely many times. 
We illustrate the approach of Fried and Schinzel~\cite{FS72} by proving the following result on quadrinomials. To that end we will need the following lemma, which was already used by Zannier~\cite{Z07}, as well as in \cite{PPS11}, to study related questions. 

\begin{lemma}[Haj\'{o}s's lemma]\label{Hajos}
Let $K$ be a field with $\textnormal{char}(K)=0$. If $g(x)\in K[x]$ with $\deg g\geq 1$ has a zero $\beta\neq 0$ of mutiplicity $m$, then $g(x)$ has at least $m+1$ terms. 
\end{lemma}

The proof of Lemma~\ref{Hajos} can be found in \cite[Sec.\@~2.6, Lemma~1]{S00}.

\begin{theorem}
Let $K$ be a field with $\textnormal{char}(K)=0$. Let $n_1>n_2>n_3$ be positive integers such that  $n_1+n_3> 2n_2$, and $a_1, a_2, a_3, a_4\in K$ such that 
$a_1a_2a_3\neq 0$. Assume that $f(x)=a_1x^{n_1}+a_2x^{n_2}+a_3x^{n_3}+a_4$ is decomposable and write $f(x)=g(h(x))$ with $g, h\in K[x]$ as in \eqref{wlog}. Then $h(x)=x^k$ and $g(x)=a_1x^{n_1/k}+a_2x^{n_2/k}+a_3x^{n_2/k}+a_4$ for some $k\in \N$ with $k\mid \gcd(n_1, n_2, n_3)$.
\end{theorem}

\begin{proof}
Let $\tilde a_i=a_i/a_1$, $i=2, 3$ and $\tilde f(x)=x^{n_1}+\tilde a_2x^{n_2}+\tilde a_3x^{n_3}$. Since $f(x)$ is decomposable, so is $\tilde f(x)$. Write $\tilde f(x)$ as in \eqref{wlog} (with $g(x)$ and $h(x)$ replaced by $\tilde g(x)$ and $\tilde h(x)$), i.e. $\tilde f(x)=\tilde g(\tilde h(x))$, $\tilde h(x)$ is monic and $\tilde h(0)=0$.
Then $\tilde g(x)$ is monic and well and $\tilde g(0)=0$. Let $\deg \tilde g=t$ and $\deg \tilde h=k$. Note that if $\tilde h(x)=x^k$ then $k\mid n_i$ for all $i=1, 2, 3$, and $\tilde g(x)=x^{n_1/k}+\tilde a_2x^{n_2/k}+\tilde a_3x^{n_3/k}$. Then also $h(x)=x^k$ and $g(x)=a_1x^{n_1/k}+a_2x^{n_2/k}+a_3x^{n_2/k}+a_4$ for some $k\in \N$ with $k\mid \gcd(n_1, n_2, n_3)$.  Assume henceforth $\tilde h(x)\neq x^k$.  Let
\[
\tilde g(x)=\prod_{i=1}^r (x-x_i)^{\alpha_i},\   \textnormal{where} \ x_i\neq x_j\  \textnormal{for}\ i\neq j, \ \textnormal{and}\  \alpha_i\in \N.
\]
Note that $\alpha_1+\cdots+\alpha_r=t$ and
\[
\tilde f(x)=\tilde g(\tilde h(x))=\prod_{i=1}^r (\tilde h(x)-x_i)^{\alpha_i},
\]
 Since $x \mid \tilde f(x)$ and $\tilde h(x)-x_i$ are relatively prime in pairs, it follows that exactly one factor, say $\tilde h(x)-x_1$, is divisible by $x$, and hence 
\[
\tilde h(x)-x_1= x^l\hat h(x),\ \textnormal{where} \ l\in \N,\  \hat h(x)\in \overline K[x],\ \hat h(0)\neq 0.
\]
 Note that from $\tilde h(0)=0$ it follows that $\tilde h(x)=x^l\hat h(x)$. Further note that 
\begin{equation}\label{tildeg}
l\alpha_1=n_3\quad \& \quad \hat h(x)^{\alpha_1} \mid x^{n_1-n_3}+\tilde a_2x^{n_2-n_3}+\tilde a_3.
\end{equation}
 If $\deg \hat h\geq 1$, from \eqref{tildeg} it follows that $x^{n_1-n_3}+\tilde a_2x^{n_2-n_3}+\tilde a_3$ has a zero of mutiplicity $\alpha_1$. From Lemma~\ref{Hajos} it follows that $\alpha_1\leq 2$. Hence, either $\hat h(x)=1$ or $\alpha_1\in \{1, 2\}$. If $\hat h(x)=1$, then $l=k$ and $\tilde h(x)=x^k$, a contradiction. Analogously,  $\hat h(x)\neq x^m$ for all $m\in \N$.  Assume henceforth $\deg \hat h\geq 1$ and $\hat h(x)\neq x^m$ for all $m\in \N$.  Since $l\alpha_1=n_3$ and $\alpha_1\in \{1, 2\}$, it follows that $l\in \{n_3, n_3/2\}$. Let 
\[
\hat h(x)=x^{m_1}+ax^{m_2} + \textnormal{lower degree terms}, \ m_1>m_2\geq 0, \ a\neq 0. 
\]
From $\tilde f(x)=\tilde g(\tilde h(x))$ it follows that
\begin{align*}
x^{n_1}+\tilde a_2x^{n_2}+\tilde a_3x^{n_3}&=\tilde g(x) \circ \left ( x^l\left (x^{m_1}+ax^{m_2} + \textnormal{lower degree terms}\right)\right)\\
&=x^{n_1}+ta x^{n_1+m_2-m_1}+\textnormal{lower degree terms},
\end{align*}
(compare with $(*)$).  Hence, $n_1+m_2-m_1=n_2$  i.e.\@ $m_1-m_2=n_1-n_2$. Then $n_1-n_2\leq m_1$. Since $m_1=k-l$, and $l\in \{n_3, n_3/2\}$, it follows that $m_1\leq (n_1-n_3)/2$, hence $2(n_1-n_2)\leq n_1-n_3$, i.e. $n_1+n_3\leq 2n_2$, a contradiction.  
\end{proof}
As already mentioned, Zannier~\cite{Z08} proved that the number of terms of $g(h(x))$ tends to infinity as the number of terms of $h(x)$ tends to infinity. In the same paper, he gave an "algorithmic" parametric description of possible decompositions $f(x)=g(h(x))$, where $f(x)$ is a polynomial with a given number of terms and $g(x)$ and $h(x)$ are arbitrary polynomials. Fuchs and Zannier~\cite{FZ12} considered lacunary rational functions $f(x)$ (expressible as the quotient of two polynomials (not necessarily coprime) having each at most a given number $\ell$ of terms). By looking at the possible decompositions $f(x) = g(h(x))$, where $g(x), h(x)$ are rational functions of degree larger than $1$, they proved that, apart from certain exceptional cases which they completely described, the degree of $g(x)$ is bounded only in terms of $\ell$ (with explicit bounds). This is a rational function analogue of Theorem~\ref{Zannier}. 
In a very recent paper ~\cite{FMZ15}, via new methods, it is proved that for completely general algebraic equations $f(x,g(x))=0$, where $f(x,y)$ is monic of arbitrary degree in $y$, and has boundedly many terms in $x$, the number of terms of $g(x)$ is bounded. This includes previous results as special cases. 

\subsection{Acknowledgements}
The authors are thankful for the support of the Austrian Science Fund (FWF) via projects W1230-N13, FWF-P24302, FWF-P26114 and F5510 (part of the Special Research Program (SFB) "Quasi-Monte Carlo Methods: Theory and Applications").

\bibliographystyle{amsplain}
\bibliography{Intro}

\end{document}